\documentclass[
final
]{dmtcs-episciences}

\usepackage[utf8]{inputenc}

\usepackage[english]{babel}
\usepackage{amsthm,amsmath,amssymb}
\usepackage[small,bf,hang]{caption} 

\usepackage{multirow}

\usepackage{algorithm}
\usepackage{algorithmic}

\usepackage{hyperref}

\newtheorem{theorem}{Theorem}
\theoremstyle{plain}
\newtheorem{lemma}[theorem]{Lemma}
\newtheorem{proposition}[theorem]{Proposition}
\newtheorem{corollary}[theorem]{Corollary}
\newtheorem{claim}[theorem]{Claim}

\graphicspath{{figures/}}

\author{Jan Goedgebeur\affiliationmark{1,2}\thanks{Both authors are supported by a Postdoctoral Fellowship of the Research Foundation Flanders (FWO).}
  \and Carol T.\ Zamfirescu\affiliationmark{1,3}
 }
\title{On almost hypohamiltonian graphs}
\affiliation{
  Department of Applied Mathematics, Computer Science and Statistics, Ghent University, Ghent, Belgium\\
  Computer Science Department, University of Mons, Mons, Belgium\\
  Department of Mathematics, Babe\c{s}-Bolyai University, Cluj-Napoca, Roumania
  }

\received{2019-3-21}
\accepted{2019-7-3}

\begin{document}
\publicationdetails{21}{2019}{4}{5}{5300}
\maketitle
\begin{abstract}
A graph $G$ is \emph{almost hypohamiltonian} (a.h.) if $G$ is non-hamiltonian, there exists a vertex $w$ in $G$ such that $G - w$ is non-hamiltonian, and $G - v$ is hamiltonian for every vertex $v \ne w$ in $G$. The second author asked in [\emph{J. Graph Theory} \textbf{79} (2015) 63--81] for all orders for which a.h.\ graphs exist. Here we solve this problem. To this end, we present a specialised algorithm which generates complete sets of a.h.\ graphs for various orders. Furthermore, we show that the smallest cubic a.h.\ graphs have order~26. We provide a lower bound for the order of the smallest planar a.h.\ graph and improve the upper bound for the order of the smallest planar a.h.\ graph containing a cubic vertex. We also determine the smallest planar a.h.\ graphs of girth~5, both in the general and cubic case. Finally, we extend a result of Steffen on snarks and improve two bounds on longest paths and longest cycles in polyhedral graphs due to Jooyandeh, McKay, {\"O}sterg{\aa}rd, Pettersson, and the second author.
\end{abstract}


\section{Introduction}
\label{section:intro}

The graphs in this paper are undirected, finite, connected, and contain neither loops nor multiple edges, unless explicitly stated otherwise. A cycle through every vertex of a graph is called \emph{hamiltonian} and a graph containing such a cycle is called \emph{hamiltonian}, as well. A graph $G$ is \emph{hypohamiltonian} if $G$ itself is non-hamiltonian, but for every vertex $v$ in $G$, the graph $G - v$ is hamiltonian. We refer the reader to the 1993 survey of Holton and Sheehan~\cite{HS93}, and, for recent results, to the article of Jooyandeh, McKay, \"{O}sterg{\aa}rd, Pettersson, and the second author~\cite{JMOPZ}.

Following Aldred, McKay, and Wormald~\cite{AMW97}, we call a graph $G$ \emph{hypocyclic} if for every $v \in V(G)$, the graph~$G - v$ is hamiltonian. Hamiltonian hypocyclic graphs are often called ``1-hamiltonian'', see e.g.~\cite{CKL70}, so the family of all hypocyclic graphs is the disjoint union of the families of all 1-hamiltonian and hypohamiltonian graphs. A graph $G$ is \emph{almost hypocyclic} if there exists a $w \in V(G)$ such that $G - w$ is non-hamiltonian, but for every vertex $v \ne w$ in $G$, the graph~$G - v$ is hamiltonian. We shall call $w$ the \emph{exceptional vertex} of $G$. If $G$ is additionally non-hamiltonian, then $G$ is called \emph{almost hypohamiltonian} (with exceptional vertex $w$).

The \emph{girth} of a graph is the length of its shortest cycle. A cycle of length~$k$ will be called a \emph{$k$-cycle}. Consider a graph~$G$. For $S \subset V(G)$, $G[S]$ shall denote the graph induced by $S$. A not necessarily connected subgraph $G' = (V(G'),E(G')) \subset G = (V(G),E(G))$ is \emph{spanning} if $V(G') = V(G)$. For a set $X$, we denote by $|X|$ its cardinality.

Let us motivate our interest in the titular family of graphs, i.e.\ almost hypohamiltonian graphs. Two motivations paved the way for other applications: firstly, although Thomassen's question from 1978~\cite{Th78} whether 4-connected hypohamiltonian graphs exist remains open, it can be shown that there are infinitely many 4-connected almost hypohamiltonian graphs~\cite{Za15}. Secondly, almost hypohamiltonian graphs can be used to construct hypohamiltonian graphs. This approach yields, as we shall see in Section~\ref{subsect:planar_case}, an alternative proof of the fact that infinitely many planar hypohamiltonian graphs exist---this was first shown by Thomassen~\cite{Th76}, while Chv\'atal~\cite{Ch73} had asked whether the statement is true, and Gr\"unbaum~\cite[p.~37]{Gr74} had conjectured that it is not.

In a different direction, using a result of Tutte, Thomassen proved in~\cite{Th78} the beautiful theorem that if all vertex-deleted subgraphs of a planar graph of minimum degree at least~4 are hamiltonian, then the graph itself must be hamiltonian. (The statement becomes untrue if ``4'' is replaced with ``3'', as proven by Thomassen~\cite{Th81}.) The second author extended this result in several directions~\cite{Z_cubic}, one of which concerns almost hypohamiltonian graphs and immediately yields, together with Thomassen's initial result, that if all but one vertex-deleted subgraphs of a planar graph of minimum degree at least~4 are hamiltonian, then the graph itself has to be hamiltonian. Inspired by another idea of Thomassen, the second author and T.~Zamfirescu~\cite{ZZ} recently solved the problem of Chv\'atal from 1973 whether \emph{any} graph can occur as an induced subgraph of some hypohamiltonian graph. For the solution, almost hypohamiltonian graphs were used---although hypohamiltonian graphs can be used as well, almost hypohamiltonian graphs yield smaller graphs with the desired properties in the planar case.

Wiener~\cite{Wi} characterised 2-fragments of hypotraceable graphs which have edge-connectivity~2 in terms of hypohamiltonian and almost hypohamiltonian graphs~\cite{Wi17}, complementing a result of Thomassen~\cite{Th74-1}. In a similar vein, historically all constructions of hypotraceable graphs (which are structurally notoriously difficult to tackle) relied on hypohamiltonian graphs as ``building blocks''.
(A graph $G$ is \textit{hypotraceable} if it is non-traceable, but $G - v$ is traceable for every $v \in V(G)$. A graph is \textit{traceable} if it contains a hamiltonian path.)
Wiener~\cite{Wi} established that almost hypohamiltonian graphs can be employed as well, and hereby improved the upper bound for the order of the smallest planar hypotraceable graph.
Motivated by a classical problem of Gallai~\cite{Ga68}, we also give---similar to Wiener's application in~\cite{Wi}---improvements of two bounds of Jooyandeh, McKay, \"{O}sterg{\aa}rd, Pettersson, and the second author~\cite{JMOPZ}, see Section~\ref{section:long}.

In order to settle two questions raised by Wiener~\cite{Wi17}, the second author recently studied non-hamiltonian graphs in which every vertex-deleted subgraph is traceable, and their relation with almost hypohamiltonian graphs~\cite{Za17}. Steffen~\cite{St98} showed that hypohamiltonian snarks are \emph{bicritical}, i.e.\ the graph itself is not 3-edge-colourable but removing any two vertices gives a 3-edge-colourable graph. We shall briefly comment in Section~\ref{subsect:cubic_case} on the fact that the same holds for almost hypohamiltonian snarks.

We emphasise that in the vast majority of the above applications, it is crucial that the almost hypohamiltonian graph has a \emph{cubic} exceptional vertex. We shall therefore pay special attention to this case, and in particular to the case when all vertices of the graph are cubic.

This article is structured as follows. In Section~\ref{section:almhypoham} we extend a series of results (and settle certain open questions) of the second author~\cite{Za15,Za16}. In particular, in Section~\ref{subsect:generation_algo} we present a specialised algorithm for generating all almost hypohamiltonian graphs. We also use our implementation of this algorithm to generate complete sets of almost hypohamiltonian graphs for various orders and in Section~\ref{subsect:gen_case} we show that there exists an almost hypohamiltonian graph of order~$n$ if and only if $n \ge 17$, solving~\cite[Problem~2]{Za15}.

In Section~\ref{subsect:cubic_case} we show that the smallest cubic almost hypohamiltonian graphs have order~26, which settles the non-planar variant of~\cite[Problem~1]{Za15}---the planar version had been solved by McKay (private communication). In Section~\ref{subsect:planar_case} we give a lower bound for the order of the smallest planar almost hypohamiltonian graphs, and improve the upper bound for the order of the smallest planar almost hypohamiltonian graph whose exceptional vertex is cubic (from 47, see~\cite{Za15}, to 36), which are particularly important due to results from~\cite{Za15}. This answers \cite[Problem~6(a)]{Za16} and \cite[Problem~7]{Za16}. We also prove that such a graph of order $n$ exists for every $n \ge 73$, partially answering \cite[Problem~6(b)]{Za16}. Furthermore, we show that the smallest planar almost hypohamiltonian graph of girth~5 has order 44. (In the hypohamiltonian case, this number is~45, see~\cite{JMOPZ}.) Finally, in Section~\ref{subsect:planar_cubic_case} we determine that the smallest planar cubic almost hypohamiltonian graph has order at least~54, and that the smallest such graph which additionally has girth~5, is of order~68. (The corresponding number for the hypohamiltonian case is~76, see McKay's paper~\cite{Mc16}.) We end this article with Section~\ref{section:long} in which we give an application of almost hypohamiltonicity to problems concerning longest paths and longest cycles in polyhedral graphs.


\section{Almost hypohamiltonian graphs}
\label{section:almhypoham}

\subsection{The generation algorithm}
\label{subsect:generation_algo}

In~\cite{GZ} we presented an algorithm for the exhaustive generation of hypohamiltonian graphs. It is based on work of Aldred, McKay, and Wormald~\cite{AMW97}. We now describe how we modified this algorithm to generate almost hypohamiltonian graphs exhaustively. We point out that in the following arguments there is some overlap with~\cite{GZ}---still, we choose to present these results for two reasons: firstly, changes with respect to~\cite{GZ} are nonetheless necessary and important, and secondly, we would like to make this paper to a large extent self-contained.

Let $G$ be a possibly disconnected graph, and $p(G)$ the minimum number of disjoint paths needed to cover all vertices of $G$. Denote by $V_1(G)$ the vertices of degree~1 in $G$, and by $I(G)$ the set of all isolated vertices and all isolated $K_2$'s (i.e.\ isolated edges together with their endpoints) of~$G$. Put
$$k(G) = \begin{cases}
                            0 & {\rm if } \ G \ {\rm is \ empty,}\\
                            \max \left\{ 1, \left\lceil{\frac{|V_1|}{2}}\right\rceil \right\}  & {\rm if } \ I(G) = \emptyset \ {\rm but} \ G \ {\rm is \ not \ empty},\\
                            |I(G)| + k(G - I(G)) & {\rm else}.
                        \end{cases}$$

We now adapt two lemmas of Aldred, McKay, and Wormald~\cite{AMW97} from the hypocyclic to the almost hypocyclic case.

\begin{lemma}\label{lem:A+B}
Given an almost hypocyclic graph $G$, for any partition $(W,X)$ of the vertices of $G$ with $|W| > 1$ and $|X| > 1$, we have that $$p(G[W]) < |X| \quad {\rm {\it and}} \quad k(G[W]) < |X|.$$
\end{lemma}

Consider a graph $G$ containing a partition $(W,X)$ of its vertices with $|W| > 1$ and $|X| > 1$. If $p(G[W]) \ge |X|$, then we call $(W,X)$ a \emph{type A obstruction}, and if $k(G[W]) \ge |X|$, then we speak of a \emph{type B obstruction}. Note that $k(G) \le p(G)$, so every type B obstruction is also a type A obstruction. However, for efficiency reasons we only consider type A obstructions where $G[W]$ is a union of disjoint paths in the algorithm. Therefore it is still useful to consider type B obstructions as well.

\begin{lemma}\label{lem:C}
Let $G$ be an almost hypocyclic graph and consider a partition $(W,X)$ of the vertices of $G$ with $|W| > 1$ and $|X| > 1$ such that $W$ is an independent set. Furthermore, for some vertex $v \in X$, define $n_1$ and $n_2$ to be the number of vertices of $X - v$ joined to one or more than one vertex of $W$, respectively. Then we have $2n_2 + n_1 \ge 2 |W|$ for every $v \in X$.
\end{lemma}

If all assumptions of Lemma~\ref{lem:C} are met and $2n_2 + n_1 < 2 |W|$ for some $v \in X$, we call $(W,X,v)$ a \emph{type C obstruction}. The following proof is a straightforward adaptation of the proof of Aldred, McKay, and Wormald~\cite{AMW97}.

\begin{proof}
\emph{(Proof of Lemmas~\ref{lem:A+B} and \ref{lem:C}).} Let $w$ be the exceptional vertex of $G$. Consider $v \in X$ with $v \ne w$ (this choice is possible, since $|X| > 1$), and let ${\mathfrak h}$ be a hamiltonian cycle in $G - v$. The number of components of ${\mathfrak h}$ restricted to $W$ must be at least $p(G[W])$, so the number of components in ${\mathfrak h}$ restricted to $X - v$ is at least $p(G[W])$ (and thus we have $|X| - 1 \ge p(G[W])$). 

Finally, consider the same cycle ${\mathfrak h}$. The number of edges of ${\mathfrak h}$ between $W$ and $X$ must be $2|W|$, but $X$ can supply at most $2n_2 + n_1$.
\end{proof}

The first part of the above proof handles type A and, since $k(G) \le p(G)$, type B obstructions. The second part deals with type C obstructions.

\begin{lemma}\label{lem:3-conn}
Let $G$ be an almost hypohamiltonian graph. Then $G$ is $3$-connected. In particular, $G$ has minimum degree~$3$.
\end{lemma}

\begin{proof}
We denote the exceptional vertex of $G$ by $w$. Since for every vertex $v \ne w$ the graph $G - v$ is hamiltonian, $G$ must be 2-connected. Now let $\{ x,y \}$ be a 2-cut in $G$. Both $G - x$ and $G - y$ are non-hamiltonian, which is impossible. So $G$ does not contain 2-cuts, thus $G$ is 3-connected.
\end{proof}

\medskip

Note that if we allow more than one exceptional vertex, 3-connectedness is not guaranteed anymore. Put rigorously: consider a 2-connected graph $G$ of circumference $|V(G)| - 1$ and let $W \subset V(G)$ be the (possibly empty) set of all vertices such that for every $w \in W$ the graph $G - w$ is non-hamiltonian. (And thus, for all $v \in V(G) \setminus W$, the graph $G - v$ is hamiltonian.) We say that $G$ is \emph{$|W|$-hypohamiltonian}. Although 0-hypohamiltonian (i.e.\ hypohamiltonian) and 1-hypohamiltonian (i.e.\ almost hypohamiltonian) graphs are 3-connected, it is easy to see that 2-hypohamiltonian graphs of connectivity~2 exist: consider $K_{2,3}$. This emphasises the intrinsic qualitative difference between hypohamiltonian and almost hypohamiltonian graphs on the one hand, and $k$-hypohamiltonian graphs with $k \ge 2$ on the other hand.

\begin{lemma}\label{lemma:c4c}
Let $M$ be a $3$-edge-cut in an almost hypohamiltonian graph $G$ with exceptional vertex $w$. Then $G - M$ contains exactly two components $A_1$ and $A_2$ with $A = K_1$ and $A_2 \ne K_1$. In particular, almost hypohamiltonian graphs are cyclically $4$-edge-connected.
\end{lemma}

\begin{proof}
Recall that a 3-edge-cut cannot separate a 3-(edge-)connected graph into more than two parts. Let $G - M$ contain components $A_1,A_2$ and assume that $A_1 \ne K_1$ and $A_2 \ne K_1$. We put $M = \{ a_1b_1, a_2b_2, a_3b_3 \}$, where $a_i \in V(A_1), b_i \in V(A_2)$ for all $i$. Since $G$ is 3-connected, the elements of the set $\{ a_1, a_2, a_3, b_1, b_2, b_3 \}$ are pairwise distinct. As $G$ is almost hypohamiltonian, there exists an $i \in \{ 1,2,3 \}$ such that $G - b_i$ and $G - a_i$ both are hamiltonian. Hence, there is a hamiltonian path in $A_1$ with end-vertices $a_j$ and $a_k$, $j \ne i \ne k$, and a hamiltonian path in $A_2$ with end-vertices $b_j$ and $b_k$, $j \ne i \ne k$. These paths together with $a_jb_j$ and $a_kb_k$ yield a hamiltonian cycle in $G$, a contradiction.
\end{proof}

\medskip

Collier and Schmeichel~\cite[p.~196]{CS78} observed that the vertices of a triangle in a hypohamiltonian graph have degree at least~4. We now show that this holds for almost hypohamiltonian graphs, as well.

\begin{lemma}\label{lemma:no_deg3_triangle}
Let $G$ be an almost hypohamiltonian graph containing a triangle $T$. Then every vertex of $T$ has degree at least~$4$.
\end{lemma}

\begin{proof}
Put $V(T) = \{ v_1, v_2, v_3 \}$, where $v_3$ shall be cubic. Since $G$ is almost hypohamiltonian, at least one of $G - v_1$ and $G - v_2$ must be hamiltonian, say $G - v_1$. Let ${\mathfrak h}$ be a hamiltonian cycle in $G - v_1$. As $v_3v_2 \in E({\mathfrak h})$, replace in ${\mathfrak h}$ the edge $v_3v_2$ with the path $v_3v_1v_2$. We obtain a hamiltonian cycle in $G$, a contradiction.
\end{proof}

\medskip

We make use of these lemmas in our algorithm to generate all almost hypohamiltonian graphs---first we give an intuitive description, which is followed by the rigorous argument. Intuitively, by a \emph{good $Y$-edge} we mean an edge which works towards the removal of a type $Y$~obstruction, where $Y \in \{ A, B, C \}$. Let us now formally define these good $Y$-edges.

Given an almost hypohamiltonian graph $G'$, and let $G$ be a spanning subgraph of $G'$ which contains a type A obstruction $(W,X)$ (we choose $W$ such that $G[W]$ is a union of disjoint paths). Since $G'$ is almost hypohamiltonian it cannot contain a type A obstruction.
Thus, there must be an edge in $E(G') \setminus E(G)$ whose endpoints are in different components of $G[W]$. We call such an edge a \textit{good $A$-edge} for $(W,X)$. Note that there must also be an edge in $E(G') \setminus E(G)$ for which at least one of the endpoints has degree at most one in $G[W]$ (but its endpoints could be in the same component of $G[W]$). Our computational experiments indicate that when generating graphs with girth at least 4, it is more efficient to define a \textit{good $A$-edge} for $(W,X)$ as an edge in $E(G') \setminus E(G)$ for which at least one of the endpoints has degree at most one in $G[W]$.
To clarify our procedure: let $\bar G$ be the graph obtained after adding a good $A$-edge to $G$. If $\bar G[W]$ contains vertices of degree 3, we remove those vertices from $W$ (and add them to $X$) for the next iteration of the algorithm in order to guarantee that $\bar G[W]$ remains a union of disjoint paths.

In the same vein, a \textit{good $B$-edge} for a type B obstruction $(W,X)$ in $G$ is a non-edge of $G$ that joins two vertices of $W$, where at least one of those vertices has degree at most one in $G[W]$. Finally, a \textit{good $C$-edge} for a type C obstruction $(W,X,v)$ in $G$ is a non-edge $e$ of $G$ for which one of the two following conditions holds:

\begin{itemize}
\item[(i)] Both endpoints of $e$ are in $W$.

\item[(ii)] One endpoint of $e$ is in $W$ and the other endpoint is in $X-v$ and has at most one neighbour in $W$.
\end{itemize}

It is straightforward that this is the only way to destroy a type B/C obstruction.

The pseudocode of our enumeration algorithm is given in Algorithm~\ref{algo:init-algo} and Algorithm~\ref{algo:construct}.
It is analogous to the algorithm we presented in~\cite{GZ} to generate all hypohamiltonian graphs. Nonetheless, we choose to present it here as well for the sake of convenience.

In order to generate all almost hypohamiltonian graphs with $n$~vertices we start from a graph $G$ which consists of an $(n-1)$-cycle (which we will denote as $C_{n-1}$) and an isolated vertex $h$ (disjoint from the cycle), so $G - h$ is hamiltonian. We then connect $h$ to $D$~vertices of the $(n-1)$-cycle in all possible ways and then perform Algorithm~\ref{algo:construct} on these graphs. Algorithm~\ref{algo:construct} will continue to recursively add edges in all possible ways such that at most one vertex $w$, the exceptional vertex, has degree larger than~$D$.

It is essential for the efficiency of the algorithm that as few as possible edges are added by Algorithm~\ref{algo:construct}, while still guaranteeing that all almost hypohamiltonian graphs are found by the algorithm.

We omit the proof of the following theorem as it is analogous to the proof of~\cite[Theorem~2.8]{GZ}.

\begin{theorem}
If Algorithm~\ref{algo:init-algo} terminates, the list of graphs $\mathcal H$ outputted by the algorithm is the list of all almost hypohamiltonian graphs with $n$ vertices.
\end{theorem}

\begin{algorithm}[ht!]
\small
\caption{Generate all almost hypohamiltonian graphs with $n$ vertices}
\label{algo:init-algo}
  \begin{algorithmic}[1]
	\STATE let $\mathcal H$ be an empty list
	\STATE let $G := C_{n-1} + h$
	\FORALL{$3 \le D \le n-1$}
		\STATE // Generate all almost hypohamiltonian graphs where at least one vertex has degree $D$ and at most one vertex has degree larger than $D$
		\FOR{every way of connecting $h$ of $G$ with $D$ vertices of the $C_{n-1}$}
		\STATE Call the resulting graph $G'$
		\STATE Construct($G', D$) \ // i.e.\ perform Algorithm~\ref{algo:construct}
		\ENDFOR
	\ENDFOR
	\STATE Output $\mathcal H$
  \end{algorithmic}
\end{algorithm}

\begin{algorithm}[ht!]
\small
\caption{Construct(Graph $G$, int $D$)}
\label{algo:construct}
  \begin{algorithmic}[1]
		\IF{$G$ is non-hamiltonian AND not generated before} \label{line:hamiltonian}
			\IF{$G$ contains a type A obstruction $(W,X)$}
				\FOR{every good A-edge $e \notin E(G)$ for $(W,X)$ for which at most one vertex has degree larger than $D$ in $G+e$}
					\STATE Construct($G+e, D$) \label{line:destroy_typeA}
				\ENDFOR
			\ELSIF{$G$ contains a vertex $v$ of degree 2}
				\FOR{every edge $e \notin E(G)$ which contains $v$ as an endpoint for which at most one vertex has degree larger than $D$ in $G+e$}
					\STATE Construct($G+e, D$) \label{line:destroy_deg2}
				\ENDFOR
			\ELSIF{$G$ contains a type C obstruction $(W,X,v)$} \label{line:typeC_obstr}
				\FOR{every good C-edge $e \notin E(G)$ for $(W,X,v)$ for which at most one vertex has degree larger than $D$ in $G+e$}
					\STATE Construct($G+e, D$)
				\ENDFOR
			\ELSIF{$G$ contains a vertex $v$ of degree 3 which is part of a triangle} \label{line:deg3_triangle}
				\FOR{every edge $e \notin E(G)$ which contains $v$ as an endpoint for which at most one vertex has degree larger than $D$ in $G+e$}
					\STATE Construct($G+e, D$)
				\ENDFOR			
			\ELSIF{$G$ contains a type B obstruction $(W,X)$} \label{line:typeB_obstr}
				\FOR{every good B-edge $e \notin E(G)$ for $(W,X)$ for which at most one vertex has degree larger than $D$ in $G+e$}
					\STATE Construct($G+e, D$)
				\ENDFOR					
			\ELSE
				\IF{$G$ is almost hypohamiltonian} \label{line:hypohamiltonian}
					\STATE add $G$ to the list $\mathcal H$
				\ENDIF
					\FOR{every edge $e \notin E(G)$ for which at most one vertex has degree larger than $D$ in $G+e$}
						\STATE Construct($G+e, D$)
					\ENDFOR	
			\ENDIF
		\ENDIF	
  \end{algorithmic}
\end{algorithm}

Note that since our algorithm only adds edges and never removes any vertices or edges, all graphs obtained by the algorithm from a graph with a $g$-cycle will have a cycle of length at most~$g$. So in case we only want to generate almost hypohamiltonian graphs with a given lower bound~$k$ on the girth, we can prune the construction when a graph with a cycle with length less than~$k$ is constructed.

For more details on the algorithm we refer to~\cite{GZ}.

\subsection{The general case}
\label{subsect:gen_case}

By using our implementation of the algorithm described in Section~\ref{subsect:generation_algo} we generated all almost hypohamiltonian graphs with girth at least~$g$ of order~$n$ for various values of $g$ and $n$. The counts of the number of almost hypohamiltonian graphs can be found in Table~\ref{table:number_of_almhypoham_graphs}. Therein, we also mention the number of almost hypohamiltonian graphs whose exceptional vertex is cubic---these are of special importance due to the fact that two such graphs can be combined to form a hypohamiltonian graph, see \cite[Theorem~3]{Za15} for details. In each case we went as far as computationally possible.

Note that we could not determine the complete set of all almost hypohamiltonian graphs with 24 vertices and girth~5---however, we succeeded in obtaining a sample, which was essential to complete the proof of Theorem~\ref{thm:all_orders}. Combining our findings with the second author's~\cite[Theorem~2]{Za15}, we have the following, settling~\cite[Problem~2]{Za15}.

\begin{theorem} \label{thm:all_orders}
An almost hypohamiltonian graph of order $n$ exists if and only if $n \ge 17$. The smallest almost hypohamiltonian graph of girth~$4$ (girth~$5$) has order~$18$ (order~$17$).
\end{theorem}

Hypohamiltonian graphs of order $n$ exist if and only if $n \in \{ 10, 13, 15, 16 \}$ or $n \ge 18$, see~\cite{AMW97}. The smallest hypohamiltonian graph of girth~4 (girth~5) has order~18 (order~10); for details we refer to~\cite{GZ}. No almost hypohamiltonian graphs of girths other than 4 or 5 are known, but it seems reasonable to believe that adapting a technique of Thomassen~\cite{Th74-2} will provide examples of girth~3. We know of the existence of hypohamiltonian graphs of girth~$g$ for every $g \in \{ 3,...,7\}$, see~\cite{GZ}, while no hypohamiltonian graphs of girth greater than~7 are known. (The smallest hypohamiltonian graph of girth~7 is Coxeter's graph.) In Figure~\ref{fig:almhypo_17} the two smallest almost hypohamiltonian graphs (each of order~17) are depicted---both have girth~5. Figure~\ref{fig:almhypo_18_g4} shows the smallest almost hypohamiltonian graph of girth~4 (which has order~18).

All graphs from Table~\ref{table:number_of_almhypoham_graphs} can be
downloaded from the \textit{House of Graphs}~\cite{hog} at
\url{http://hog.grinvin.org/AlmostHypohamiltonian} and also be inspected in the database of interesting graphs by searching for the keywords ``almost hypohamiltonian''.

\begin{table}[htb!]
\centering
\small
	\begin{tabular}{|c || c | c | c | c |}
		\hline
		Order & \# alm hypoham. & $g \geq 4$ & $g \geq 5$ & $g \geq 6$\\
		\hline
$0-16$  &  0  &  0  &  0  &  0  \\
17  &  2 {\small(1)}  &  2 (1)  &  2 (1)  &  0  \\
18  &  2 (1) &  2 (1)  &  1 (1)  &  0  \\
19  &  ?  &  27 (11)  &  4 (4)  &  0  \\
20  &  ?  &  ?  &  14 (6)  &  0  \\
21  &  ?  &  ?  &  27 (18) &  0 \\
22  &  ?  &  ?  &  133 (86) &  0  \\
23  &  ?  &  ?  &  404 (161)  &  0  \\
24  &  ?  &  ?  &  $\ge 68$  &  0   \\
25  &  ?  &  ?  &  ?  &  0   \\
26  &  ?  &  ?  &  ?  &  0   \\
		\hline
	\end{tabular}
	\caption{The number of almost hypohamiltonian graphs. The columns with a header of the form $g \geq k$ contain the number of almost hypohamiltonian graphs with girth at least~$k$. In parentheses the number of almost hypohamiltonian graphs which have a cubic exceptional vertex are given.}
	\label{table:number_of_almhypoham_graphs}
\end{table}

\begin{figure}[h!t]
	\centering
	\includegraphics[width=0.4\textwidth]{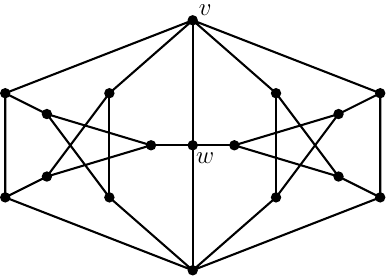}
	\caption{The almost hypohamiltonian graph of order~17 from~\cite{Za15}. Its exceptional vertex is~$w$. By deleting the edge $vw$, we obtain the smallest almost hypohamiltonian graph both in terms of order and size.}
	\label{fig:almhypo_17}
\end{figure}

\begin{figure}[h!t]
	\centering
	\includegraphics[trim = 0mm 0mm 22mm 16mm, clip,width=0.5\textwidth]{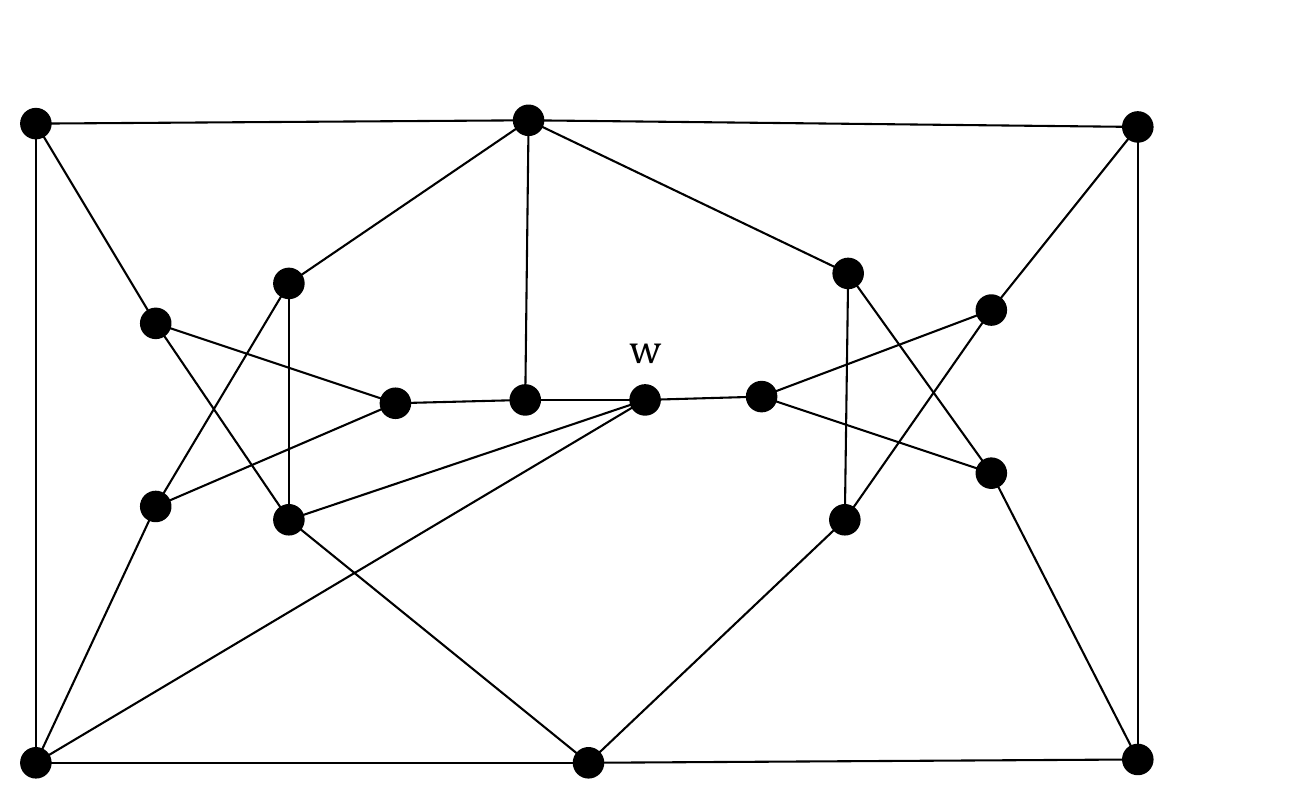}
	\caption{The smallest almost hypohamiltonian graph of girth~4. It has 18 vertices. Its exceptional vertex is~$w$.}
	\label{fig:almhypo_18_g4}
\end{figure}

Our generator for almost hypohamiltonian graphs was obtained by extending the code of our generator for hypohamiltonian graphs from~\cite{GZ}. In that paper we describe how we extensively tested the correctness of our implementation. We also used multiple independent programs to test hamiltonicity and almost hypohamiltonicity---one of those programs was kindly provided to us by Gunnar Brinkmann---and in each case the results were in complete agreement. Furthermore, the source code of our new generator can be downloaded and inspected at~\cite{genhypo-site}.

\subsection{The cubic case}
\label{subsect:cubic_case}

The algorithm described in Section~\ref{subsect:generation_algo} can also be used to only generate cubic almost hypohamiltonian graphs, but it is much more efficient to use a generator for cubic graphs and test which graphs are almost hypohamiltonian as a filter.

We used the program \textit{snarkhunter}~\cite{brinkmann_11,BG15} to generate all cubic graphs with girth at least~4 up to 32~vertices and tested them for almost hypohamiltonicity. (Note that by Lemma~\ref{lemma:no_deg3_triangle} cubic almost hypohamiltonian graphs must have girth at least~4.) The results are presented in Table~\ref{table:counts_almhypo_cubic}. The smallest cubic almost hypohamiltonian graphs of girth~4 and~5 are shown in Figure~\ref{fig:cubic_almhypo}.

We also used the complete lists of snarks from~\cite{snark-paper} to determine all almost hypohamiltonian snarks up to 36~vertices. A \textit{snark} is a cyclically 4-edge-connected cubic graph with girth at least~5 which is not 3-edge-colourable. The number of almost hypohamiltonian snarks is listed in the last column of Table~\ref{table:counts_almhypo_cubic}. The sudden increase at 34~vertices is striking and also occurs for hypohamiltonian snarks as described in~\cite{snark-paper}.


\begin{figure}[h!t]
    \centering
    \includegraphics[trim = 0mm 0mm 22mm 16mm, clip, width=0.35\textwidth]{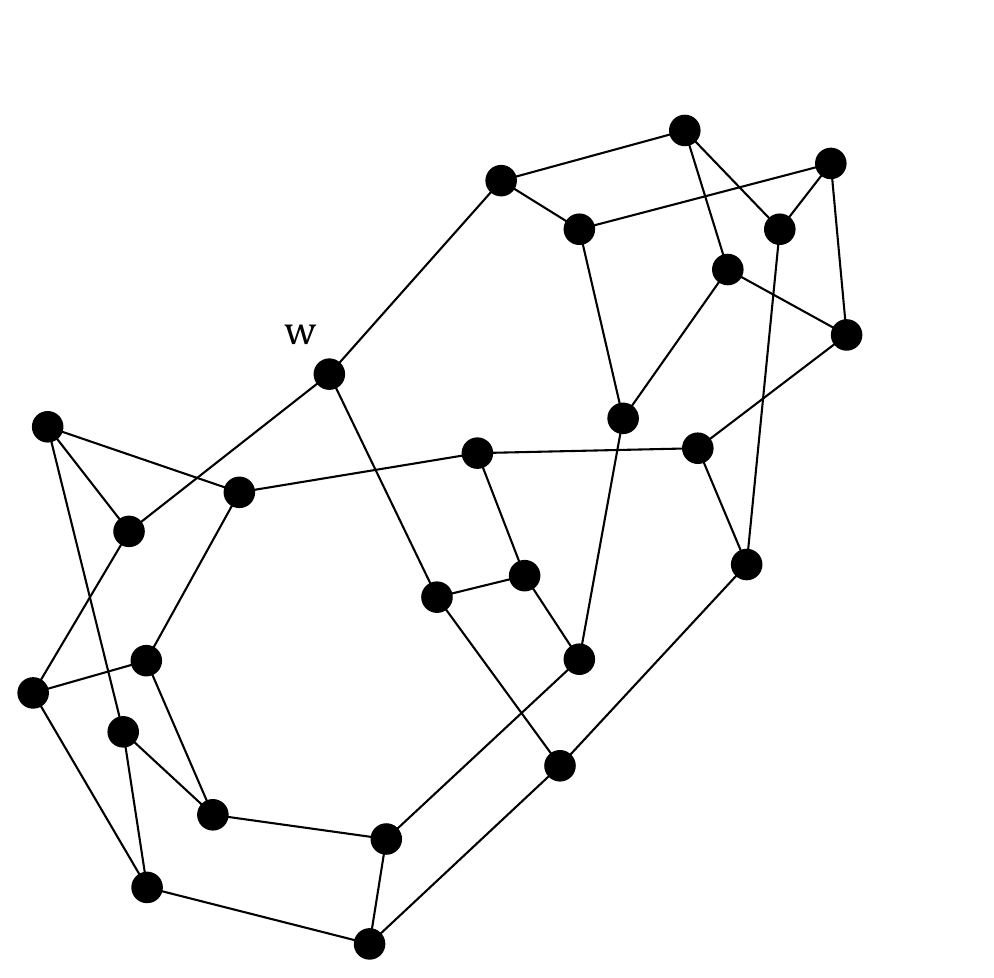}
    \qquad \qquad
	\includegraphics[trim = 0mm 0mm 22mm 16mm, clip, width=0.35\textwidth]{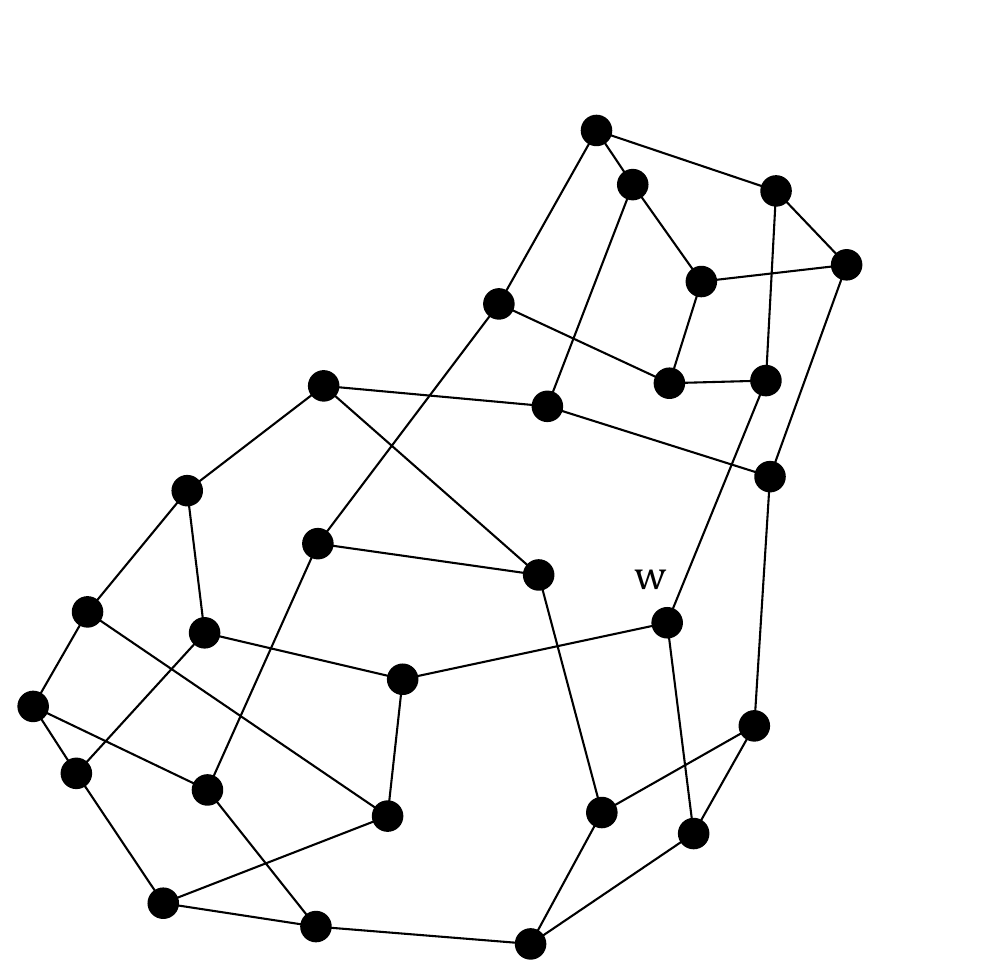}
    \caption{The left-hand figure depicts one of the ten smallest cubic almost hypohamiltonian graphs of girth~5. It has 26~vertices. The right-hand figure shows one of the four smallest cubic almost hypohamiltonian graphs of girth~4. It has 28~vertices. In both figures the exceptional vertex is labelled $w$.}
    \label{fig:cubic_almhypo}
\end{figure}

All almost hypohamiltonian graphs from Table~\ref{table:counts_almhypo_cubic} can be downloaded from the \textit{House of Graphs}~\cite{hog} at
\url{http://hog.grinvin.org/AlmostHypohamiltonian} and also be inspected in the database of interesting graphs by searching for the keywords ``almost hypohamiltonian''.

Using \textit{snarkhunter} we also verified that there are no cubic almost hypohamiltonian graphs with girth 6 up to at least 36 vertices.

By combining the results from Table~\ref{table:counts_almhypo_cubic} with the results on non-hamiltonian cubic graphs from~\cite{GZ} we obtain the following bounds for cubic almost hypohamiltonian graphs.

\begin{theorem}
Let $c_g$ denote the order of the smallest cubic almost hypohamiltonian graph with girth~$g$. We have $$c_4 = 28, \quad c_5 = 26, \quad c_6 \ge 38, \quad c_7 \ge 44, \quad c_8 \ge 50, \quad {\rm {\it and}} \quad c_9 \ge 66.$$
\end{theorem}

In comparison, if we denote by $c'_g$ the order of the smallest cubic hypohamiltonian graph with girth~$g$, we have the following. (See \cite{GZ} for more details.)
$$c'_4 = 24, \quad c'_5 = 10, \quad c'_6 = c'_7 = 28, \quad c'_8 \ge 50, \quad {\rm and} \quad c'_9 \ge 66.$$


\begin{table}[htb!]
\centering
\footnotesize
	\begin{tabular}{|c || r | r | r | r | r || r |}
		\hline
		\multirow{2}{*}{Order} & \multirow{2}{*}{$g \ge 4$} & Non-ham. & Alm. & Alm.\ hypo. & Alm.\ hypo. & Alm.\ hypo.\\		
		 &  & and $g \ge 4$ & hypo. & and $g \ge 5$ & and $g \ge 6$ & snark\\				
		\hline
6  &  1  &  0  &  0  &  0  &  0 & 0 \\
8  &  2  &  0  &  0  &  0  &  0 & 0 \\
10  &  6  &  1  &  0  &  0  &  0 & 0 \\
12  &  22  &  0  &  0  &  0  &  0 & 0\\
14  &  110  &  2  &  0  &  0  &  0 & 0\\
16  &  792  &  8  &  0  &  0  &  0 & 0\\
18  &  7 805  &  59  &  0  &  0  &  0 & 0\\
20  &  97 546  &  425  &  0  &  0  &  0  & 0\\
22  &  1 435 720  &  3 862  &  0  &  0  &  0 & 0 \\
24  &  23 780 814  &  41 293  &  0  &  0  &  0 & 0 \\
26  &  432 757 568  &  518 159  &  10  &  10  &  0 & 10 \\
28  &  8 542 471 494  &  7 398 734  &  6  &  2  &  0 & 2 \\
30  &  181 492 137 812  &  117 963 348  &  25  &  12  &  0 & 11 \\
32  &  4 127 077 143 862  &  2 069 516 990  &  74  &  4  &  0 & 0 \\
34  &  ?  &  ?  &  ?  &  ?  &  ? & 6 253 \\
36  &  ?  &  ?  &  ?  &  ?  &  ? & 2 243 \\
		\hline
	\end{tabular}

\caption{Counts of almost hypohamiltonian graphs among cubic graphs. $g$ stands for girth. Note that due to Lemma~\ref{lemma:no_deg3_triangle}, the girth of a cubic almost hypohamiltonian graph is at least~4. The last column denotes the number of almost hypohamiltonian snarks.}
\label{table:counts_almhypo_cubic}

\end{table}

Steffen~\cite{St98} showed that cubic hypohamiltonian graphs with chromatic index~4 are bicritical, while Nedela and \v{S}koviera~\cite{NS96} proved that cubic bicritical graphs are necessarily cyclically 4-edge-connected and have girth at least~5. Thus, cubic hypohamiltonian graphs with chromatic index~4 are snarks. (Not every cubic hypohamiltonian graph is a snark, since planar cubic hypohamiltonian graphs exist, as shown by Thomassen~\cite{Th81}, and the Four Colour Theorem is equivalent to the statement that no snark is planar.) As Nedela and \v{S}koviera~\cite{NS96} observe, since removing a single vertex from a cubic graph with no 3-edge-colouring cannot yield a 3-edge-colourable graph, hypohamiltonian snarks lie on the border between cubic graphs that are 3-edge-colourable and those that are not. The same holds in fact for almost hypohamiltonian snarks, as well: complementing a result of Steffen~\cite{St98} and using his proof idea, one can show the following.

\begin{proposition}
Every almost hypohamiltonian snark is bicritical.
\end{proposition}


Thus, by the aforementioned result of Nedela and \v{S}koviera~\cite{NS96}, we obtain that each almost hypohamiltonian cubic graph with chromatic index~4 has girth at least~5 and cyclic edge-connectivity at least~4, and thus is a snark.

\subsection{The planar case}
\label{subsect:planar_case}

Until recently, the smallest known planar almost hypohamiltonian graph had order~39, see~\cite{Za15}. \cite[Problem~3]{Za15} and \cite[Problem~6]{Za16} ask for the smallest order of a planar almost hypohamiltonian graph, and the smallest number $n_0$ such that there exists a planar almost hypohamiltonian graph of order~$n$ for every $n \ge n_0$. Although we are not able to fully settle these questions, we can report progress on both the lower and upper bounds of the order of the smallest planar almost hypohamiltonian graph, as well as an improvement of~$n_0$.

We shall strengthen the second author's~\cite[Theorem 1]{Za15} and \cite[Theorem 4]{Za15}. Although Wiener recently has shown~\cite{Wi} the existence of a planar almost hypohamiltonian graph of order~31 (which constitutes the currently best upper bound for the order of the smallest planar almost hypohamiltonian graph), the planar almost hypohamiltonian graphs we will construct in Theorem~\ref{plalmhypo}---the smallest among them has 36 vertices---have a cubic exceptional vertex, while the exceptional vertex in Wiener's 31-vertex graph has degree~4. Cubic exceptional vertices are of particular interest due to the fact that we can combine two almost hypohamiltonian graphs, each having a cubic exceptional vertex, into a hypohamiltonian graph, see~\cite[Theorem 3]{Za15}---no method is known to do this if the exceptional vertex is non-cubic.

Let us first focus on the lower bound for the order of the smallest planar almost hypohamiltonian graph. Since the algorithm for generating all almost hypohamiltonian graphs presented in Section~\ref{subsect:generation_algo} only adds edges and never removes any vertices or edges, all graphs obtained by the algorithm from a non-planar graph will remain non-planar. So in case we only want to generate planar almost hypohamiltonian graphs, we can prune the construction when a non-planar graph is constructed. We used Boyer and Myrvold's algorithm~\cite{boyer2004cutting} to test if a graph is planar.

Following this approach, we showed that the smallest planar almost hypohamiltonian graph has at least 22 vertices and that the smallest planar almost hypohamiltonian graph of girth~4 must have at least 26 vertices. For girth~5 it turned out to be more efficient to use the program \textit{plantri}~\cite{brinkmann_07} instead. Using \textit{plantri} we generated all planar 3-connected graphs with girth~5 up to 48~vertices and tested them for almost hypohamiltonicity. This yielded the following results.

\begin{theorem}\label{thm:planar}
The smallest planar almost hypohamiltonian graph with girth~$5$ has $44$~vertices. There is exactly one such graph of that order and it is shown in Figure~\ref{fig:planar_almhypo_44}. There are exactly $42$~planar almost hypohamiltonian graphs with girth~$5$ on $47$~vertices. These are the only planar almost hypohamiltonian graphs with girth~$5$ up to $48$ vertices.
\end{theorem}

\medskip

The graphs from Theorem~\ref{thm:planar} can be downloaded from the database of interesting graphs at the \textit{House of Graphs}~\cite{hog} by searching for the keywords ``planar almost hypohamiltonian''.

\begin{figure}[h!t]
	\centering
	\includegraphics[width=0.4\textwidth]{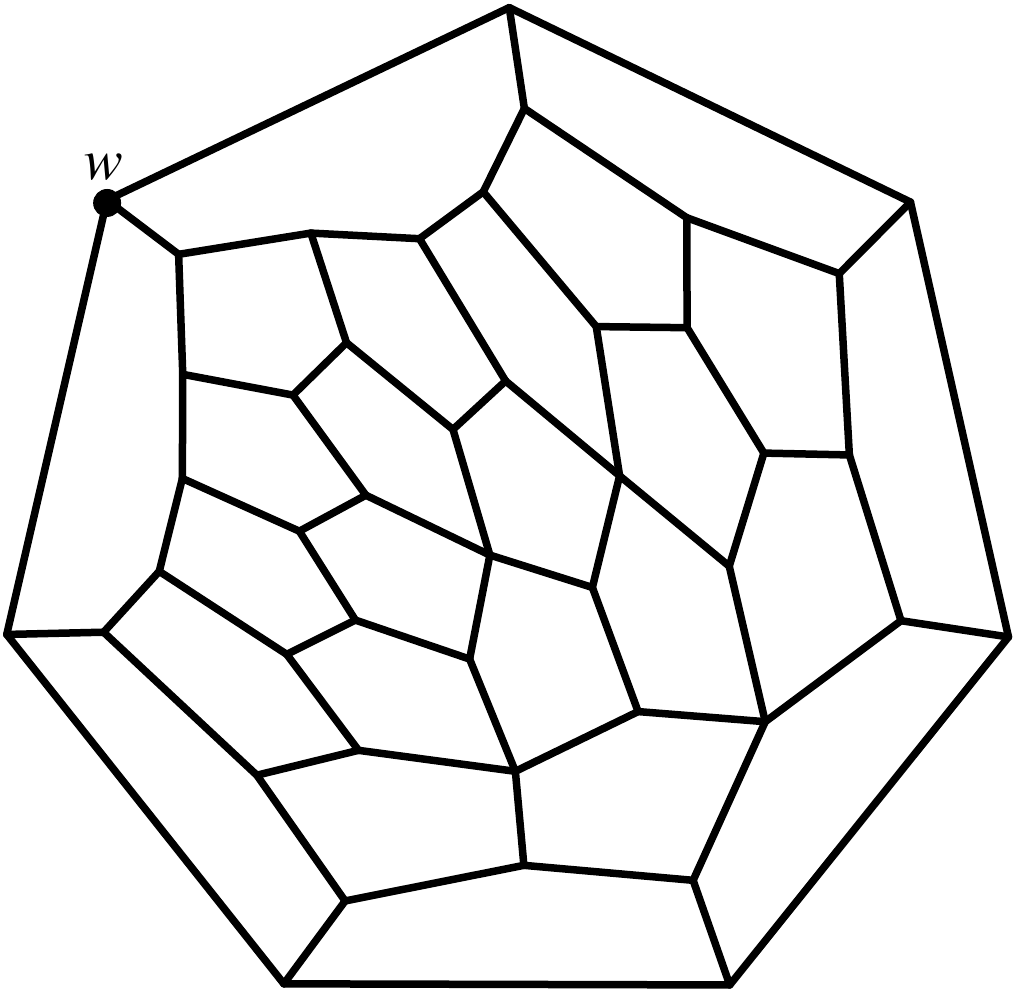}
	\caption{The smallest planar almost hypohamiltonian graph of girth~5. It has 44 vertices. Its exceptional vertex is~$w$.}
	\label{fig:planar_almhypo_44}
\end{figure}

In summary, we have the following lower bounds.

\begin{corollary}
Let $\bar{h}$ ($\bar{h}_g$) denote the order of the smallest planar almost hypohamiltonian graph (of girth~$g$). We have $$\bar{h} \ge 22, \quad \bar{h}_4 \ge 26, \quad {\rm {\it and}} \quad \bar{h}_5 = 44.$$
\end{corollary}

Wiener~\cite{Wi} showed that there exists a planar almost hypohamiltonian graph of order~31 and girth~4, so $\bar{h} \le \bar{h}_4 \le 31$. (No planar almost hypohamiltonian graph of girth~3 is known. Adapting a technique of Thomassen~\cite{Th74-2} could provide examples of girth~3, as mentioned earlier.) In comparison, if we denote by $\bar{h}'$ ($\bar{h}'_g$) the order of the smallest planar hypohamiltonian graph (of girth~$g$), we have the following---more details can be found in~\cite{GZ} and, for the girth~3 case, in~\cite{ZZ}. $$23 \le \bar{h}' \le 40, \quad 23 \le \bar{h}'_3 \le 216, \quad 27 \le \bar{h}'_4 \le 40, \quad {\rm and} \quad \bar{h}'_5 = 45.$$

We now prepare for showing that there exists a planar almost hypohamiltonian graph of order~36 whose exceptional vertex is cubic, as well as improving $n_0$ (defined in the first paragraph of this subsection) from 76 (see~\cite{Za15}) to 73. Let $G$ be a graph containing a $4$-cycle $v_1 v_2 v_3 v_4 = C$. We denote by ${\rm Th}(G_C)$ the graph obtained from $G$ by deleting the edges $v_1 v_2$ and $v_3 v_4$ and adding a 4-cycle $v'_1 v'_2 v'_3 v'_4$ disjoint from $G$, and the edges $v_i v'_i$, $1 \le i \le 4$, to $G$. (The operation Th was introduced by Thomassen~\cite{Th81} to show that there exist infinitely many planar cubic hypohamiltonian graphs.) The statement of Lemma~\ref{Thom-lemma} is a slight modification of a claim of Thomassen, see~\cite{Th81}; we skip its proof.

\begin{lemma}[Thomassen~\cite{Th81}]\label{Thom-lemma}
Let $G$ be a plane non-hamiltonian graph containing a quadrilateral face bounded by the cycle $C$. Then ${\rm Th}(G_C)$ is planar and non-hamiltonian.
\end{lemma}

In a plane graph, we call a face \emph{cubic} if all of its vertices are cubic. In~\cite{Th81}, Thomassen also showed that given a plane hypohamiltonian graph $G$ which contains a cubic quadrilateral face bounded by the cycle $C$, we have that ${\rm Th}(G_C)$ is a planar hypohamiltonian graph, as well. We require a modified version of Thomassen's result.

\begin{lemma}[Zamfirescu~\cite{Za15}]\label{almhypo-quad}
Let $G$ be a plane almost hypohamiltonian graph with exceptional vertex $w$, and let $G$ contain a cubic quadrilateral face bounded by the cycle $C$ such that $w \notin V(C)$. Then ${\rm Th}(G_C)$ is a planar almost hypohamiltonian graph, as well.
\end{lemma}

Consider graphs $G$ and $H$, and the cubic vertices $x \in V(G)$ and $y \in V(H)$. Denote by $G_x H_y$ one of the graphs obtained from
$G - x$ and $H - y$ by identifying the vertices in $N(x)$ with those in $N(y)$ using a bijection (so $G_x H_y$ has $|V(G)| + |V(H)| - 5$ vertices). Thomassen~\cite{Th74-1} showed that if $G$ and $H$ are hypohamiltonian, then $G_x H_y$ is hypohamiltonian, as well. Note that if $G$ is hypohamiltonian, then $G$ contains no triangle with a cubic vertex (as shown by Collier and Schmeichel~\cite[p.~196]{CS78}). We will need the following lemma.

\begin{lemma}[Zamfirescu~\cite{Za15}]\label{hypo+almhypo}
Let $G$ be an almost hypohamiltonian graph which contains a cubic vertex $x$ different from the exceptional vertex $w$ of $G$, and $H$ a hypohamiltonian graph which contains a cubic vertex $y$. Then $G_x H_y$ is an almost hypohamiltonian graph with exceptional vertex $w$. If $G$ and $H$ are planar, then so is $G_x H_y$.
\end{lemma}

A further ingredient is the following result of Grinberg.

\begin{theorem}[Grinberg's Criterion~\cite{Gr68}]\label{Grinberg}
Given a plane graph with a hamiltonian cycle ${\mathfrak h}$ and exactly $f_i$ ($f'_i$) $i$-gons inside (outside) of ${\mathfrak h}$, we have $$\sum\limits_{i \ge 3} (i-2)(f_i - f'_i) = 0.\qquad (\dagger)$$
\end{theorem}

\begin{theorem}\label{plalmhypo}
There exists a planar almost hypohamiltonian graph of order $n$ and with cubic exceptional vertex for (i) $n = 36 + 4k$, $k \ge 0$, and (ii) every $n \ge 73$.
\end{theorem}

\begin{proof}
(i) We first show that the graph $G$ from Figure~\ref{fig:planar_alm_G} is almost hypohamiltonian. By Grinberg's Criterion, $G$ is non-hamiltonian, since all faces of $G$ are pentagons with exactly one exception, which is a quadrilateral---now the left-hand side of ($\dagger$) cannot vanish. $G - w$ is non-hamiltonian, as well: denote by $Q = a'b'c'd'$ the quadrilateral in $G$ such that $a'$ is the cubic vertex of $Q$, and let $w$ be the exceptional vertex of $G$, as shown in Figure~\ref{fig:planar_alm_G} (left-hand side). Assume that there exists a hamiltonian cycle ${\mathfrak h}$ in $G - w$. By Grinberg's Criterion, in $G - w$, $Q$ and the 9-gon $N$ must lie on the same side of ${\mathfrak h}$. But this is impossible, as $a' \in V({\mathfrak h})$ and thus, $Q$ and $N$ lie on different sides of ${\mathfrak h}$. Hence, $G - w$ is non-hamiltonian.

\begin{figure}[h!t]
    \centering
\includegraphics[width=0.4\textwidth]{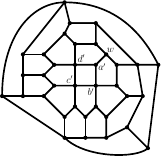}
\qquad \qquad
\includegraphics[width=0.4\textwidth]{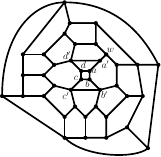}
    \caption{$G$ and $G'$ (left-hand side and right-hand side, respectively), two planar almost hypohamiltonian graphs used in the proof of Theorem~\ref{plalmhypo}. Their respective cubic exceptional vertices are labelled $w$.}
    \label{fig:planar_alm_G}
\end{figure}

We skip the straightforward verification that for every vertex $v \ne w$, the graph $G - v$ is hamiltonian. Thus, we have shown that there exists a planar almost hypohamiltonian graph of order~36 whose exceptional vertex is cubic. Note that we cannot apply Lemma~\ref{almhypo-quad} to $G$ since $a'b'c'd'$ is not a cubic quadrilateral, thus we require the following claim.


\begin{claim}\label{claim}
The graph $G'$ from Figure~\ref{fig:planar_alm_G} (right-hand side) is planar and almost hypohamiltonian with exceptional vertex $w$.
\end{claim}



\noindent \emph{Proof of the Claim.} Since $G' = {\rm Th}(G_Q)$, by Lemma~\ref{Thom-lemma} the graph $G'$ is planar and non-hamiltonian. Let ${\mathfrak h}$ be a hamiltonian cycle in $G' - w$. Certainly, $aa'd' \subset {\mathfrak h}$. If $dd' \in E({\mathfrak h})$, then $b'baa'd'dcc' \subset {\mathfrak h}$, which replaced with $b'a'd'c'$ gives a hamiltonian cycle in $G - w$, a contradiction, so $dd' \notin E({\mathfrak h})$. Hence $a'adcbb' \subset {\mathfrak h}$, which replaced with $a'b'$ gives a hamiltonian cycle in $G - w$, once more a contradiction. 

It is quickly verified that there exists a hamiltonian cycle ${\mathfrak h}$ in $G - a'$ containing the path $b'c'd'$. Deleting from ${\mathfrak h}$ the edge $c'd'$ we obtain the path ${\mathfrak p}$ in $G'$. Now ${\mathfrak p} \cup d'dabcc'$ is a hamiltonian cycle in $G' - a'$. Finding hamiltonian cycles in $G' - b'$, $G' - c'$, and $G' - d'$ works in a similar (and in fact even simpler) fashion. ${\mathfrak p} \cup d'a'adcc'b'$ is a hamiltonian cycle in $G' - b$, and ${\mathfrak p} \cup d'a'abcc'b'$ is a hamiltonian cycle in $G' - d$. It is easy to see that there exists a hamiltonian cycle ${\mathfrak h}'$ in $G - d'$ containing the edge $b'a'$. Deleting this edge from ${\mathfrak h}'$ we obtain the path ${\mathfrak p}'$ in $G'$. Now ${\mathfrak p}' \cup b'bcdd'a'$ is a hamiltonian cycle in $G' - a$, and ${\mathfrak p}' \cup b'badd'a'$ is a hamiltonian cycle in $G' - c$.

It remains to show that $G' - v'$ is hamiltonian for every $v' \in V(G') \setminus \{ w, a, a', b, b',$ $c, c', d, d' \}$. Let $v$ be the vertex in $G$ corresponding to $v'$ in $G'$. Let ${\mathfrak h}''$ be a hamiltonian cycle in $G - v$. If $a'b' \in E({\mathfrak h}'')$, replace $a'b'$ with $a'adcbb'$. If $a'b' \notin E({\mathfrak h}'')$, then necessarily $a'd' \in E({\mathfrak h}'')$, which we replace with $a'abcdd'$. In both cases we have obtained a hamiltonian cycle in $G' - v'$, which finishes the proof of Claim~\ref{claim}.

\smallskip

Using Claim~\ref{claim} and iterating Lemma~\ref{almhypo-quad} ad infinitum, the proof of (i) is complete.

\smallskip

(ii) In \cite{JMOPZ}, it was shown that for every $n \ge 42$ there exists a planar hypohamiltonian graph of order $n$, which we will call $H^n$. For arbitrary but fixed $n \ge 42$, let $x \in V(H^n)$ be cubic (Thomassen showed that every planar hypohamiltonian graph contains a cubic vertex~\cite{Th78}), and let $y \in V(G)$ be cubic, where $G$ is the graph shown in Figure~\ref{fig:planar_alm_G} (left-hand side). Applying Lemma~\ref{hypo+almhypo}, we obtain that $H^n_x G_y$ is an almost hypohamiltonian graph with a cubic exceptional vertex.
\end{proof}



\medskip

By applying~\cite[Theorem~3]{Za15}, Theorem~\ref{plalmhypo}~(i) yields an alternative proof of the fact that infinitely many planar hypohamiltonian graphs exist. (This was first shown by Thomassen~\cite{Th76} in order to settle a question of Chv\'atal~\cite{Ch73}.)

\subsection{The planar cubic case}
\label{subsect:planar_cubic_case}

Using the program \textit{plantri}~\cite{brinkmann_07} we generated all planar cyclically 4-edge-connected cubic graphs with girth at least~4 up to 52~vertices and tested them for almost hypohamiltonicity. No such graphs were found, so we have (recall that, by Lemma~\ref{lemma:c4c}, almost hypohamiltonian graphs are cyclically 4-edge-connected):

\begin{theorem}
The smallest planar cubic almost hypohamiltonian graph has at least $54$~vertices.
\end{theorem}

\medskip


With \textit{plantri} we also generated all planar cyclically 4-edge-connected cubic graphs with girth~5 up to 78~vertices. This yielded the following results. These results were independently obtained by McKay (private communication) but were not published.

\begin{theorem}\label{thm:planar_cubic}
The smallest planar cubic almost hypohamiltonian graph of girth~$5$ has $68$~vertices. There are exactly three such graphs of that order and they are shown in Figure~\ref{fig:planar_cubic_g5}. There are exactly $81$~planar cubic almost hypohamiltonian graphs with girth~$5$ on $74$~vertices. These are the only planar cubic almost hypohamiltonian graphs with girth~$5$ up to at least $78$~vertices.
\end{theorem}

\medskip

The graphs from Theorem~\ref{thm:planar_cubic} can be downloaded from the database of interesting graphs at the \textit{House of Graphs}~\cite{hog} by searching for the keywords ``planar cubic almost hypohamiltonian''.


\begin{figure}[h!t]
    \centering
    \includegraphics[width=0.29\textwidth]{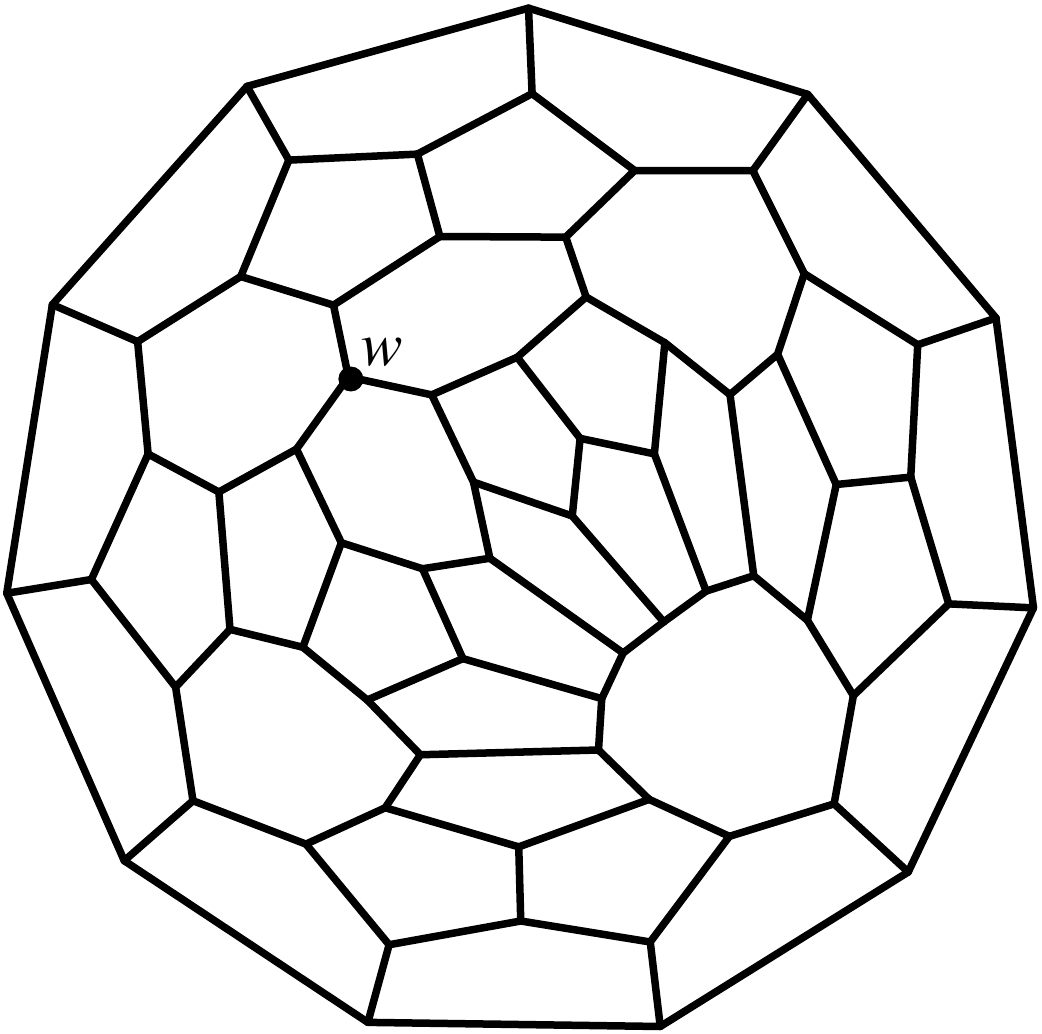} \quad
    \includegraphics[width=0.29\textwidth]{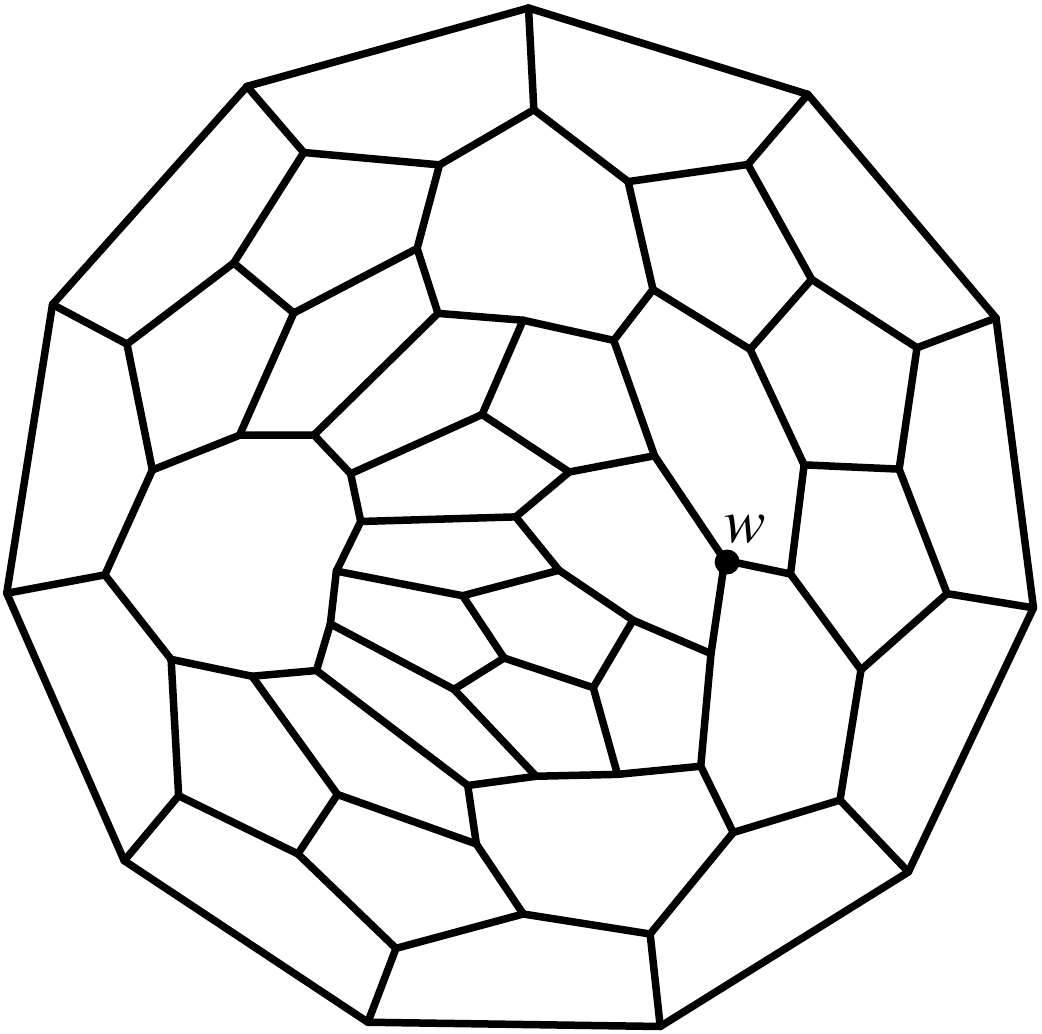} \quad
    \includegraphics[width=0.29\textwidth]{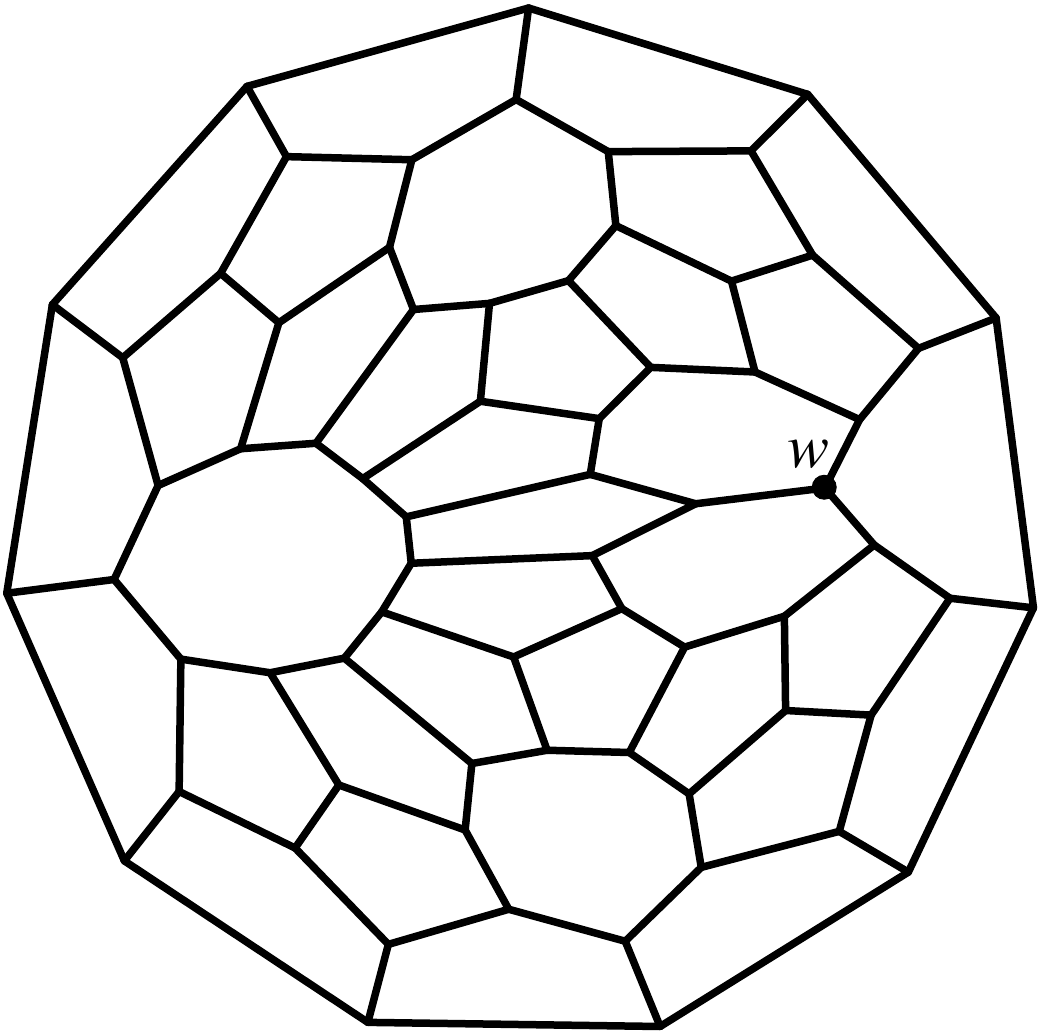}
    \caption{All three smallest planar cubic almost hypohamiltonian graphs of girth~$5$. Each has 68~vertices. In each figure the exceptional vertex is labelled~$w$.}
    \label{fig:planar_cubic_g5}
\end{figure}

The smallest planar cubic hypohamiltonian graph of girth~5 has order~76, as proven by McKay~\cite{Mc16}. In~\cite{GZ}, we showed that there exists exactly one such graph of order~78. Interestingly, McKay's three extremal graphs of order~76 have trivial automorphism group, whereas our 78-vertex example has $D_{3h}$ symmetry (as an abstract group, this is the dihedral group of order~12). Furthermore, it is noteworthy that despite the fact that the smallest almost hypohamiltonian graph is significantly larger than the smallest hypohamiltonian graph (17 versus 10 vertices), in the planar and planar cubic case, the current bounds strongly indicate that the roles are reversed.

\section{On longest paths and longest cycles}
\label{section:long}

In this section we present an application of our results on almost hypohamiltonian graphs. Motivated by Gallai's classical problem whether in every connected graph there is a vertex lying on all longest paths~\cite{Ga68}, T.~Zamfirescu~\cite{Za76} introduced and studied the following numbers: denote with $C_k^j \left(P_k^j \right)$ the smallest order of a $k$-connected graph in which any $j$ vertices are avoided by some longest cycle (longest path). $\overline{C}_k^j$ and $\overline{P}_k^j$ denote the respective numbers for the planar case. A survey of related results is given in~\cite{SZZ13}---we restrict ourselves here to giving only the following tabular overview.

\begin{table}[htbp]
\centering
\begin{tabular}{c|c c c c c c}
$j,k$ & 1,1 & 1,2 & 1,3 & 2,1 & 2,2 & 2,3\\
\hline
$C_k^j$ & \underline{6} & \underline{10} & \underline{10} & \underline{9} & 75 & 75\\
$P_k^j$ & \underline{12} & 26 & 36 & 93 & 270 & 270\\
$\overline{C}_k^j$ & \underline{6} & \underline{15} & 40 & \underline{9} & 135 & \textbf{2345}\\
$\overline{P}_k^j$ & 17 & 32 & 156 & 308 & 914 & \textbf{9246}\\
\end{tabular}
\caption{Upper bounds for $C_k^j$, $P_k^j$, $\overline{C}_k^j$, $\overline{P}_k^j$, where $k \in \{ 1,2,3\}$ and $j \in \{ 1,2 \}$. Underlined entries are optimal. Entries in bold are proven in this article.}
\end{table}

By using an insertion technique of T.~Zamfirescu~\cite{Za76} and the 36-vertex planar almost hypohamiltonian graph which has a cubic exceptional vertex given in Section~\ref{subsect:planar_case} (and found independently by Wiener~\cite{Wi}), Wiener~\cite{Wi} improved the upper bound for the smallest planar hypotraceable graph from 154 (proven in~\cite{JMOPZ}) to 138. Using the same approach, we can prove the following.

\newpage

\begin{theorem}
We have $\overline{C}^2_3 \le 2345$ and $\overline{P}^2_3 \le 9246.$
\end{theorem}

\noindent \emph{Proof.} Let $G$ be the planar almost hypohamiltonian graph on 36~vertices which has a cubic exceptional vertex. For the first inequality, insert $G$ into the 70-vertex planar cubic hypohamiltonian graph $\Gamma$ constructed by Araya and Wiener in~\cite{AW11}. ($\Gamma$ is the smallest known planar cubic hypohamiltonian graph.) This means that each vertex of~$\Gamma$ is replaced by $G$ minus its exceptional vertex, i.e.: consider a vertex $v$ in $\Gamma$ and let $w$ be the (cubic) exceptional vertex of $G$. We then consider the disjoint graphs $\Gamma - v$ and $G - w$ and join by adding three edges the vertices in $N(v)$ and $N(w)$, through a bijection. This is done for all vertices of $\Gamma$. We denote the resulting graph by $G'$.

Araya and Wiener proved~\cite{AW11} (using a computer) that every pair of edges in $\Gamma$ is missed by a longest cycle. Combining this fact with the almost hypohamiltonicity of $G$ and the hypohamiltonicity of $\Gamma$, we obtain that in $G'$ any pair of vertices is avoided by a longest cycle. We do not lose this property if all edges originally belonging to $\Gamma$ are contracted. By construction, the order of $G'$ is $(36 - 1) \cdot 70 = 2450$. Since $|E(\Gamma)| = 105$, after contracting all edges originally belonging to $\Gamma$, we obtain $\overline{C}^2_3 \le 2450 - 105 = 2345$.

For the second inequality, consider the graph $\Gamma$ from above and insert $\Gamma$ into $K_4$ to obtain $H$. Now insert $G$ into $H$. Finally, contract all edges which originally belonged to $H$. Since $|V(H)| = (70 - 1) \cdot 4 = 276$ and $H$ is cubic, we have that $|E(H)| = 414$. Then $\overline{P}^2_3 \le (36 - 1) \cdot 276 - 414 = 9246$. \hfill $\Box$

\bigskip

For a more detailed proof of the above bounds, replace in \cite[Corollary~3.6]{AW11} the 42-vertex planar hypohamiltonian graph with a 36-vertex planar almost hypohamiltonian graph. Observe that at no point the hamiltonicity of $G$ minus the deleted vertex is used, so hypohamiltonian graphs and almost hypohamiltonian graphs with a cubic exceptional vertex are equally useful in this context.


\acknowledgements
\label{sec:ack}
This work was carried out using the Stevin Supercomputer Infrastructure at Ghent University. 
We thank Gunnar Brinkmann for providing us with an independent program for testing almost hypohamiltonicity and Brendan McKay for his advice on \textit{plantri}.


\bibliographystyle{plain}
\bibliography{references}
\label{sec:biblio}


\end{document}